\documentclass[11pt,a4paper,reqno]{amsart}
\usepackage[english]{babel}
\usepackage[applemac]{inputenc}
\usepackage[T1]{fontenc}
\usepackage{palatino}
\usepackage{esint}
\usepackage{verbatim}
\usepackage{amsmath}
\usepackage{amssymb}
\usepackage{amsthm}
\usepackage{amsfonts}
\usepackage[pdftex]{graphicx}
\usepackage{mathtools}
\usepackage{enumitem}

\DeclarePairedDelimiter\floor{\lfloor}{\rfloor}
\usepackage[colorlinks = true, citecolor = black]{hyperref}
\title[Dorronsoro's theorem]{Dorronsoro's theorem in Heisenberg groups}
\author{Katrin F\"assler and Tuomas Orponen}
\address{Department of Mathematics\\ University of Fribourg \\ Chemin du Mus\'{e}e 23,
CH-1700 Fribourg, Switzerland}
\address{University of Helsinki, Department of Mathematics and Statistics}
\thanks{K.F.\ is supported by the Swiss National Science Foundation via the project \emph{Intrinsic rectifiability and mapping theory on the Heisenberg group}, grant No.  161299. T.O. is supported by the Academy of Finland via the project \emph{Quantitative rectifiability in Euclidean and non-Euclidean spaces}, grant No. 309365.}
\email{tuomas.orponen@helsinki.fi}
\email{katrin.faessler@unifr.ch}

\newcommand{\R}{\mathbb{R}}

\newcommand{\He}{\mathbb{H}}
\newcommand{\N}{\mathbb{N}}

\newcommand{\C}{\mathbb{C}}

\newcommand{\Z}{\mathbb{Z}}

\newcommand{\calJ}{\mathcal{J}}

\newcommand{\calD}{\mathcal{D}}

\newcommand{\calB}{\mathcal{B}}

\newcommand{\calS}{\mathcal{S}}

\newcommand{\calA}{\mathcal{A}}

\newcommand{\Rea}{\operatorname{Re}}

\newcommand{\sub}{\bigtriangleup_{\He}}

\numberwithin{equation}{section}

\theoremstyle{plain}
\newtheorem{thm}[equation]{Theorem}

\newtheorem{lemma}[equation]{Lemma}

\newtheorem{cor}[equation]{Corollary}
\newtheorem{proposition}[equation]{Proposition}

\theoremstyle{definition}

\newtheorem{definition}[equation]{Definition}

\theoremstyle{remark}
\newtheorem{remark}[equation]{Remark}

\newcommand{\nref}[1]{(\hyperref[#1]{#1})}

\begin{document}

\begin{abstract} A theorem of Dorronsoro from the 1980s quantifies the fact that real-valued Sobolev functions on Euclidean spaces can be approximated by affine functions almost everywhere, and at all sufficiently small scales. We prove a variant of Dorronsoro's theorem in Heisenberg groups: functions in horizontal Sobolev spaces can be approximated by affine functions which are independent of the last variable.

As an application, we deduce new proofs for certain vertical vs. horizontal Poincar\'e inequalities for real-valued functions on the Heisenberg group, originally due to Austin-Naor-Tessera and Lafforgue-Naor.
\end{abstract}

\maketitle

\section{Introduction}

We start with a word on general notation. For $n\in \mathbb{N}$, we consider the \emph{$n$-th Heisenberg group} $\mathbb{H}^n = (\mathbb{R}^{2n+1},\cdot)$. Points in $\mathbb{H}^n$ will typically be denoted by $x=(z,t)\in \mathbb{R}^{2n}\times \mathbb{R}$, and we
write $X_1,\ldots,X_{2n}$ for the left-invariant vector fields with the property that $X_j(0)$ is the standard $j$-th basis vector, $j\in \{1,\ldots,2n\}$. We will make no notational distinction between vector fields and the associated differential operators. If the derivatives $X_j f$, $j\in \{1,\ldots,2n\}$, of a function $f:\He^n \to \R$ exist in the distributional sense, we denote by $\nabla_{\He}f= (X_1f,\ldots, X_{2n}f)$ the \emph{horizontal gradient} of $f$. 
%Integration on $\He^n$ will always be with respect to Lebesgue measure. 
The symbol $B(x,r)$ stands for an open ball with center $x$ and radius $r$ with respect to the Kor\'{a}nyi metric on $\He^n$.

Let $f \colon \He^{n} \to \R$ be a locally integrable function, let $d \in \{0,1\}$, and let $\calA_{d}$ be the family of (real) polynomials $\R^{2n} \to \R$ of degree at most $d$. Some of the definitions below would, formally, make sense for all $d \in \N$, but only the cases $d \in \{0,1\}$ will be considered in the paper. Slightly abusing notation, we often view the elements of $\calA_{d}$ as functions on $\He^{n} = \R^{2n} \times \R \to \R$ depending only on the first $2n$ variables. For $x \in \He^{n}$, $r > 0$, and $d \in \{0,1\}$, we define the following quantity:
\begin{equation}\label{betaDefinition} \beta_{f,d}(B(x,r)) := \fint_{B(x,r)} |f(y) - A^{d}_{x,r}(y)| \, dy. \end{equation}
Here "$dy$" refers to integration with respect to Lebesgue measure on $\R^{2n + 1}$, and $A^{d}_{x,r}$ is the unique map in $\calA_{d}$ with the property that
\begin{equation}\label{affApprox} \int_{B(x,r)} (f(y) - A^{d}_{x,r}(y))A(y) \, dy = 0, \qquad A \in \calA_{d}. \end{equation}
If $f \in L^{2}(B(x,r))$, then $A^{d}_{x,r}$ is simply the orthogonal projection in $L^{2}(B(x,r))$ to the subspace $\calA_{d}$, and for instance
\begin{displaymath} A_{x,r}^{0} \equiv \langle f \rangle_{B(x,r)}, \end{displaymath}
the $L^{1}$-average of $f$ over $B(x,r)$. It is not hard to compute $A_{x,r}^{1}$ explicitly either, see \eqref{A1}, and this is one way to establish the existence of $A^{1}_{x,r}$ in the generality of $f \in L^{1}_{loc}(\He^{n})$. The uniqueness of $A_{x,r}^{d}$ is straightforward: if $A_{1},A_{2} \in \calA_{d}$ are two candidates satisfying \eqref{affApprox}, then $A_{1} - A_{2}$ is $L^{2}(B(x,r))$-orthogonal to $\calA_{1} \cap L^{2}(B(x,r))$, and hence the projection of $A_{1} - A_{2}$ to $\calA_{1} \cap L^{2}(B(x,r))$ is zero. On the other hand, since $A_{1} - A_{2} \in \calA_{1}$, this projection equals $A_{1} - A_{2}$, and so $A_{1} = A_{2}$.

Here is the main result:
\begin{thm}\label{main} Let $1 < p < \infty$, and let $f \in L^{p}(\He^{n})$ be a function with $\nabla_{\He}f \in L^{p}(\He^{n})$. Define the following square function:
\begin{displaymath} Gf(x) := \left(\int_{0}^{\infty} \left[ \tfrac{1}{r}\beta_{f,1}(B(x,r)) \right]^{2} \, \frac{dr}{r} \right)^{1/2}. \end{displaymath}
Then,
\begin{displaymath} \|Gf\|_{L^{p}(\He^{ n})} \lesssim \|\nabla_{\He} f\|_{L^{p}(\He^{n})}. \end{displaymath}
\end{thm}

The theorem applies in particular to compactly supported Lipschitz functions on $\He^n$. It is an $\He^{n}$ variant of a theorem of Dorronsoro \cite[Theorem 2]{MR796440} from the 80's, and we will simply implement his proof strategy in $\He^{n}$. As in \cite{MR796440}, we will derive Theorem \ref{main} from a more general statement, which concerns affine approximation of functions in Sobolev spaces of fractional order, as defined by Folland \cite{MR0494315}. The definition is based on the fractional (sub-)Laplace operator
\begin{displaymath} f \mapsto (-{\sub}_{,p})^{\alpha}f, \qquad \alpha\in \mathbb{C}, \: f \in \mathrm{Dom}((-{\sub}_{,p})^{\alpha})\subset L^p(\mathbb{H}^n), \end{displaymath}
see \cite[p. 181,186]{MR0494315} (where the same object is denoted by $\calJ^{\alpha}$). For the definition of Sobolev spaces below, we would only need to consider $\alpha\geq 0$, but complex values of $\alpha$ will make a brief appearance in the proof of Lemma \ref{interpolationLemma}.

%\begin{definition} Let $\calD$ be the family of Schwartz functions on $\He^{n}$. Let $1 < p < \infty$ and $\alpha \geq 0$. We define $S^{p}_{\alpha}$ to be the completion of $\calD$ under the norm
%\begin{equation}\label{norm} \|f\|_{p,\alpha} := \|(-\sub)^{\alpha/2}(f)\|_{L^{p}(\He^{n})}. \end{equation}
%Here $(-\sub)^{\alpha/2}$ is the fractional sub-Laplacian of order $\alpha/2$ on $\He^{n}$, see \cite[p. 181]{MR0494315} (where the same object is denoted by $\calJ^{\alpha/2}$).
%\end{definition}

\begin{definition}\label{d:Sobolev}
Let $1 < p < \infty$ and $\alpha \geq 0$. The (inhomogeneous) \emph{Sobolev space} $S^{p}_{\alpha}$ is $\mathrm{Dom}((-{\sub}_{,p})^{\alpha/2})$ equipped with the norm $$\|f\|_{p,\alpha}:=\|f\|_{L^{p}(\He^{n})} + \|(-{\sub}_{,p})^{\alpha/2}(f)\|_{L^{p}(\He^{n})}.$$
\end{definition}

%\begin{remark}
%We note that if $1 < p < \infty$, then $\calD$ is indeed contained in the domain of the operator $f \mapsto (-\sub)^{\alpha/2}f \in L^{p}$, and \eqref{norm} defines a norm on $\calD$. For proofs (in a more general setting), see \cite[Theorem 4.3.6(g)]{MR3469687} and \cite[Lemma 4.4.11(1)]{MR3469687}. In the literature, $S^{p}_{\alpha}$ is typically denoted by $\dot{S}^{p}_{\alpha}$, but since here we have no use for the inhomogeneous Sobolev spaces -- for which the symbol $S^{p}_{\alpha}$ is often reserved -- we keep the simpler notation. \end{remark}

In the sequel, we omit the subscript ``$p$'' in the notation for the operator $(-{\sub}_{,p})^{\alpha/2}$ whenever the meaning is clear from the context (in particular, if $(-{\sub}_{,p})^{\alpha/2}$ acts on a function $f \in S^{p}_{\alpha}$).

\begin{remark} The properties of the spaces $(S^p_{\alpha},\|\cdot\|_{p,\alpha})$ are discussed in detail in \cite[Section 4]{MR0494315} and, in a more general setting, in \cite[Section 4.4.1]{MR3469687}. In particular, $(S^p_{\alpha},\|\cdot\|_{p,\alpha})$ is a Banach space for $1 < p < \infty$ and $\alpha \geq 0$.  \end{remark}

Here is the generalised version of Theorem \ref{main}:

%Before stating the general version of Theorem \ref{main}, we also define
%\begin{displaymath} \beta_{f,0}(B(x,r)) := \fint_{B(x,r)} |f(y) - \langle f \rangle_{B(x,r)}| \, dy, \qquad x \in \He^{n}, \: r > 0, \end{displaymath}
%where $\langle f \rangle_{B(x,r)}$ is the average of $f \in L^{1}(B(x,r))$.

\begin{thm}\label{mainGeneral} Let $1 < p < \infty$ and $0 < \alpha < 2$. For $f\in S^p_{\alpha}$, define the following square function:
\begin{displaymath} G_{\alpha}f(x) := \left( \int_{0}^{\infty}  \left[ \tfrac{1}{r^{\alpha}} \beta_{f,\floor{\alpha}}(B(x,r)) \right]^{2} \, \frac{dr}{r} \right)^{1/2}, \qquad x \in \He^{n}. \end{displaymath}
Here $\floor{\alpha}$ stands for the integer part of $\alpha$. Then, $G_{\alpha}f \in L^{p}(\He^{n})$ with
\begin{equation}\label{eq:general_est} \|G_{\alpha}f\|_{L^{p}(\He^{n})} \lesssim \|(-\sub)^{\alpha/2}(f)\|_{L^{p}(\He^{n})}.  \end{equation}
\end{thm}

\begin{remark} We note that Theorem \ref{main} follows from the case $\alpha = 1$ of Theorem \ref{mainGeneral}, because $\|(-\sub)^{1/2}(f)\|_{L^{p}(\He^{n})} \sim \|\nabla_{\He} f\|_{L^{p}(\He^{n})}$ by \cite[(52)]{MR1828225}. \end{remark}

\subsection{Extensions and applications}\label{exAndAppl} Generalising the "$L^{1}$-based" numbers $\beta_{f,d}(B(x,r))$ defined in \eqref{betaDefinition}, one can consider the $L^{q}$-variants
\begin{displaymath} \beta_{f,d,q}(B(x,r)) = \left(\fint_{B(x,r)} |f(y) - A_{x,r}^{d}(y)|^{q} \, dy \right)^{1/q}, \qquad 1 \leq q < \infty.  \end{displaymath}
It is then possible to ask if and when Theorems \ref{main} and \ref{mainGeneral} continue to hold for these $\beta$-numbers. We do not here pursue the most general results: we only show in Section \ref{extensions} that Theorem \ref{main} holds for the numbers $\beta_{f,1,q}(B(x,r))$ if $1 < p < \infty$ and
\begin{displaymath} 1 \leq q < \min\left\{\frac{pQ}{Q - p},\frac{2Q}{Q - 2} \right\}, \end{displaymath}
where $Q = 2n + 2$. In particular, this range covers the case $p = q = 2$ which appears to be relevant for applications of Dorronsoro's theorem -- at least in Euclidean space, see for instance \cite[Section 10]{DS1}. The argument required for the extension is virtually the same as employed by Dorronsoro in \cite[Section 5]{MR796440}: one can literally reduce matters to Theorem \ref{main}. We will repeat the details in Section \ref{extensions}.

An explicit application in Heisenberg groups where the cases $q > 1$ come handy are certain \emph{horizontal vs. vertical Poincar\'e} inequalities, first established by Austin, Naor, and Tessera \cite{MR3095705}, later extended by Lafforgue and Naor \cite{MR3273443} and Naor and Young \cite{NY}. It turns out that many "non-endpoint" cases of these inequalities can be obtained as corollaries of Dorronsoro's theorem in the Heisenberg group. The matter will be further discussed in Section \ref{HvsVSection}.

We conclude the introduction with a few words on the proof structure of Theorem \ref{mainGeneral}. Section \ref{s:prelim} is mostly preparatory; notably, it reduces the "homogeneous" inequality \eqref{eq:general_est} to its "inhomogeneous" analogue, see Lemma \ref{l:inhomog_homog}. In Section \ref{alphaLessThanOne}, we prove Theorem \ref{mainGeneral} in the regime $0 < \alpha < 1$. Then, in Section \ref{alphaMoreThanOne}, we prove the case $1 < \alpha < 2$ by a reduction to the case $0 < \alpha < 1$. Finally, in Section \ref{interpolation}, we derive the case $\alpha = 1$ by complex interpolation.

Our proof strategy of Theorem \ref{mainGeneral} -- hence Theorem \ref{main} -- is exactly the same as in Dorronsoro's original work \cite{MR796440}. The main point here is to check that the use of horizontal Sobolev spaces in $\He^{n}$ produces no serious complications. The case $0<\alpha<1$ of Theorem \ref{mainGeneral} is essentially contained in \cite[p. 291 ff]{MR1828225}; the cases $1\leq \alpha <2$ involve approximation by polynomials of degree $d=1$, and these are not discussed in \cite{MR1828225}.

\section{Preliminaries}\label{s:prelim} We start by verifying that $A_{x,r}^{d}$ is always a near-optimal choice for the $L^{1}(B(x,r))$-approximation of $f$ by functions in $\calA_{d}$:
\begin{lemma}\label{lemma1} Let $x \in \He^{n}$, $r > 0$ and $f \in L^{1}(B(x,r))$. Then, for $d \in \{0,1\}$,
\begin{displaymath} \int_{B(x,r)} |f(y) - A^{d}_{x,r}(y)| \, dy \lesssim \int_{B(x,r)} |f(y) - A(y)| \, dy, \qquad A \in \calA_{d}. \end{displaymath}
\end{lemma}

For later use, we separately mention the following immediate corollary:
\begin{cor}\label{betaMonotonicity} Let $x_{1},x_{2} \in \He^{n}$ and $0 < r_{1} < r_{2} < \infty$ be such that $B(x_{1},r_{1}) \subset B(x_{2},r_{2})$ and $r_{2} \leq Cr_{1}$. Then
\begin{displaymath} \beta_{f,1}(B(x_{1},r_{1})) \lesssim_{C} \beta_{f,1}(B(x_{2},r_{2})). \end{displaymath}
\end{cor}

\begin{proof} Apply the lemma with $A_{x,r} = A_{x_{1},r_{1}}$ and $A = A_{x_{2},r_{2}}$. \end{proof}

\begin{proof}[Proof of Lemma \ref{lemma1}] We first claim that
\begin{equation}\label{claim} \|A_{x,r}^{d}\|_{L^{\infty}(B(x,r))} \lesssim \fint_{B(x,r)} |f(y)| \, dy. \end{equation}
One easily reduces to the case $B(x,r) = B(0,1) = B(1)$ by scalings and (left) translations, and using the uniqueness of the element $A \in \calA_{d}$ satisfying \eqref{affApprox}. Further, the equation
\begin{displaymath} \int_{B(1)} A^{d}_{0,1}(y)^{2} \, dy = \int_{B(1)} f(y) \cdot A^{d}_{0,1}(y) \, dy \end{displaymath}
follows from the definition \eqref{affApprox} of $A^{d}_{0,1}$, by choosing $A = A^{d}_{0,1}$. Consequently, using also the equivalence of all norms on the finite-dimensional space $\calA_{d}$, we infer that
\begin{displaymath} \|A^{d}_{0,1}\|^{2}_{L^{\infty}(B(1))} \sim \|A^{d}_{0,1}\|_{L^{2}(B(1))}^{2} \leq \|f\|_{L^{1}(B(1))} \|A^{d}_{0,1}\|_{L^{\infty}(B(1))}. \end{displaymath}
Rearranging the terms gives \eqref{claim}. We learned this quick argument from a paper of Prats, see \cite[Remark 2.4]{MR3677864}.

To complete the proof of the lemma, we write $A^{d}_{x,r}(g)$ for the function in $\calA_{d}$ corresponding to $g \in L^{1}(B(x,r))$, as in \eqref{betaDefinition}-\eqref{affApprox}. In particular, $g \mapsto A^{d}_{x,r}(g)$ is linear, and $A^{d}_{x,r}(A) = A$ for all $A \in \calA_{d}$. It follows that,
\begin{align*} \int_{B(x,r)} |f(y) - A^{d}_{x,r}(y)| \, dy & \leq \int_{B(x,r)} |f(y) - A(y)| \, dy + \int_{B(x,r)} |A(y) - A^{d}_{x,r}(y)| \, dy\\
& = \int_{B(x,r)} |f(y) - A(y)| \, dy + \int_{B(x,r)} |A^{d}_{x,r}(A - f)(y)| \, dy\\
& \lesssim \int_{B(x,r)} |f(y) - A(y)| \, dy, \end{align*}
using \eqref{claim} in the last inequality.
\end{proof}

Before the next remark, we record that $A^{1}_{x,r}(g)$ (as in the proof above) has the form
\begin{equation}\label{A1} A^{1}_{x,r}(g)(y) = b + \sum_{j = 1}^{2n} a_{j}(y_{j} - x_{j}), \qquad y = (y_{1},\ldots,y_{2n + 1}) \in \He^{n}, \end{equation}
where the coefficients $a_{j}$, $1 \leq j \leq 2n$, and $b$ are
\begin{displaymath} a_{j} = \frac{\fint_{B(x,r)} g(y)(y_{j} - x_{j}) \, dy}{\fint_{B(x,r)} (y_{j} - x_{j})^{2} \, dy} \quad \text{and} \quad b = \langle g \rangle_{B(x,r)}. \end{displaymath}
\begin{remark}\label{smoothApproximation} We discuss the sufficiency to prove Theorem \ref{mainGeneral} for $f$ in $\calD$, the space of smooth compactly supported functions on $\He^n$. By \cite[Theorem (4.5)]{MR0494315}, the functions in $\calD$ are dense in $S^{p}_{\alpha}$ for all $1 < p < \infty$ and $\alpha \geq 0$. So, if $f \in S^{p}_{\alpha}$, with $1 < p < \infty$ and $0 < \alpha < 2$, we may choose a sequence $\{f_{j}\}_{j \in \N} \subset \calD$ such that $\|f_{j} - f\|_{p,\alpha} \to 0$. In particular, $\|f_{j} - f\|_{L^{p}(\He^{n})} \to 0$, which easily implies that
\begin{displaymath} \beta_{f_{j},\floor{\alpha}}(B(x,r)) \to \beta_{f,\floor{\alpha}}(B(x,r)), \qquad x \in \He^{n}, \: r > 0, \: 0 < \alpha < 2, \end{displaymath}
as $j \to \infty$. To see this for $1 < \alpha < 2$, use the explicit expression for the maps $A^{1}_{x,r}$ obtained above (for $0 < \alpha < 1$ the claim is trivial, as $A^{0}_{x,r} \equiv \langle f \rangle_{B(x,r)}$). Then, if Theorem \ref{mainGeneral} has already been proved for some fixed $0 < \alpha < 2$, and for all functions in $\calD$, we infer from Fatou's lemma that
\begin{displaymath} \|G_{\alpha}f\|_{L^{p}(\He^{n})} \leq \liminf_{j \to \infty} \|G_{\alpha}f_{j}\|_{L^{p}(\He^{n})} \lesssim \liminf_{j \to \infty} \|(-\sub)^{\alpha/2}(f_j)\|_{L^{p}(\He^{n})} = \|(-\sub)^{\alpha/2}(f)\|_{L^{p}(\He^{n})}. \end{displaymath}
Hence, Theorem \ref{mainGeneral} follows for general $f \in S^{p}_{\alpha}$. \end{remark}

We conclude this section with one more reduction in the proof of Theorem \ref{mainGeneral}. The proof given below for Theorem \ref{mainGeneral} will initially produce the estimate
\begin{displaymath}
\|G_{\alpha}f\|_{L^{p}(\He^{n})} \lesssim \|f\|_{p,\alpha} = \|f\|_{L^{p}(\He^{n})} + \|(-\sub)^{\alpha/2}(f)\|_{L^{p}(\He^{n})},\quad f\in \calD,
\end{displaymath}
which is seemingly weaker than \eqref{eq:general_est}. (To be precise, this phenomenon will only occur in the case $\alpha = 1$, but that is the case most relevant for Theorem \ref{main}.) This is precisely the result Dorronsoro proves in \cite[Theorem 2]{MR796440}. However, the following homogeneity considerations allow us to remove the term $\|f\|_{L^{p}(\He^{n})}$ from the estimate.
\begin{lemma}\label{l:inhomog_homog}
Let $1 < p < \infty$ and $\alpha > 0$. If
\begin{equation}\label{inhom} \|G_{\alpha}f\|_{L^{p}(\He^{n})} \lesssim \|f\|_{p,\alpha},\quad f \in \calD,\end{equation}
then also
\begin{equation}\label{hom} \|G_{\alpha}f\|_{L^{p}(\He^{n})} \lesssim \|(-\sub)^{\alpha/2} f\|_{L^{p}(\He^{n})},\quad f \in \calD,\end{equation}
\end{lemma}

\begin{proof}
Given $f \in \calD$ and $s > 0$, set $f_s:= f \circ \delta_s \in \calD$, where $\delta_s$ denotes the usual Heisenberg dilation. Since the transformation $x \mapsto \delta_{s}(x)$ has Jacobian $s^{Q}$, with $Q = 2n + 2$, one has
\begin{equation}\label{eq:goal}
\|f_s\|_{L^p(\mathbb{H}^n)} = s^{-Q/p} \|f\|_{L^p(\mathbb{H}^n)}, \qquad s > 0.
\end{equation}
Below, we will moreover argue that
\begin{equation}\label{eq:goal 2} \begin{cases} \|G_{\alpha}f_s\|_{L^{p}(\He^{n})} =  s^{\alpha-Q/p}  \|G_{\alpha}f\|_{L^{p}(\He^{n})}, \text{ and} \\
\|(-\sub)^{\alpha/2} f_s\|_{L^{p}(\He^{n})} = s^{\alpha-Q/p}  \|(-\sub)^{\alpha/2} f\|_{L^{p}(\He^{n})}. \end{cases}
\end{equation}
Consequently, by \eqref{inhom},
\begin{align*} \|G_{\alpha}f\|_{L^{p}(\He^{n})} = s^{Q/p - \alpha}\|G_{\alpha}f_{s}\|_{L^{p}(\He^{n})} & \lesssim s^{Q/p - \alpha}[\|f_{s}\|_{L^{p}(\He^{n})} + \|(-\sub)^{\alpha/2}f_s\|_{L^{p}(\He^{n})}]\\
& = s^{-\alpha}\|f\|_{L^{p}(\He^{n})} + \|(-\sub)^{\alpha/2}f\|_{L^{p}(\He^{n})},  \end{align*}
and \eqref{hom} then follows by letting $s\to \infty$. To establish \eqref{eq:goal 2}, we first need to compute $G_{\alpha} f_s$, and for this purpose, we need expressions for the numbers $\beta_{f_s,d}$, with $d\in \mathbb{N}\cup \{0\}$ and $s > 0$. We observe that
\begin{displaymath}
A^d_{x,r}(f_s)(\delta_{1/s}(y))= A^d_{x,sr}(f)(y),\quad y\in \mathbb{H}^n.
\end{displaymath}
Again applying the change-of-variables formula, we find that
\begin{equation}\label{eq:beta1}
\beta_{f_s,d}(B(x,r)) = \beta_{f,d}(B(\delta_s(x),sr)).
\end{equation}
Thus, for all  $x\in \mathbb{H}^n$ and $\alpha, s >0$, one has
\begin{align*} G_{\alpha} f_s(x) & = \left(\int_0^{\infty} \left[\frac{1}{r^{\alpha}}\beta_{f_s,\floor{\alpha}}(B(x,r))\right]^2 \frac{dr}{r}\right)^{\frac{1}{2}}\\
&= s^{\alpha}\left(\int_0^{\infty} \left[\frac{1}{(sr)^{\alpha}}\beta_{f,\floor{\alpha}}(B(\delta_s(x),sr))\right]^2 \frac{sdr}{sr}\right)^{\frac{1}{2}}\\
&= s^{\alpha} G_{\alpha} f(\delta_s(x)).
\end{align*}
and now the first part of \eqref{eq:goal 2} follows from \eqref{eq:goal}.

Finally, to compute $ \|(-\sub)^{\alpha/2}(f_s)\|_{L^{p}(\He^{n})}$, we have to use the definition of the fractional sub-Laplacian \cite[p. 181]{MR0494315}, namely
%. For $\alpha=1$, we have $k= 1$, and the formula reduces to
\begin{equation}\label{eq:formula}
(-\sub)^{\alpha/2}(f) = \lim_{\epsilon \to 0}\frac{1}{\Gamma(k-\frac{\alpha}{2})} \int_{\epsilon}^{\infty} t^{k-(\alpha/2)-1}(-\triangle_{\mathbb{H}})^k H_t f \;dt,
\end{equation}
where $H_t f := f\ast h_t$ denotes convolution with the heat kernel $h_t(x):= h(x,t)$ and $k=\floor{\alpha/2}+1$. We recall from \cite[p.184 and Theorem 4.5]{MR0494315} that $\calD \subset \mathrm{Dom}((-{\sub}_{,p})^{\alpha/2})$ for all $1 < p < \infty$, and hence $(-{\sub}_{,p})^{\alpha/2}(f) = (-{\sub}_{,q})^{\alpha/2}(f)$ for $1<p,q<\infty$ and $f\in \calD$. Consequently, the limit in \eqref{eq:formula} exists in $L^p(\He^n)$ for any choice of $1<p<\infty$.
Exploiting the homogeneity
\begin{displaymath}
h(rx,r^2 t) = r^{-Q} h(x,t),
\end{displaymath}
we derive
\begin{align*}
H_t f_s(x) &= f_s \ast h_t(x) = \int h_t(y^{-1} x) f(\delta_s(y)) dy= \frac{1}{s^Q} \int h_t\left((\delta_{1/s}(y))^{-1} x\right)f(y) dy\\
&= \int h_{s^2 t} (y^{-1} \delta_s(x))  f(y) dy \\
&=( H_{s^2 t} f)(\delta_s(x)).
\end{align*}
Therefore,
\begin{displaymath}
(-\triangle_{\mathbb{H}})^k H_t f_s = s^{2k} \left[(-\triangle_{\mathbb{H}})^kH_{s^2 t} f\right] \circ \delta_s
\end{displaymath}
by the homogeneity of $\triangle_{\mathbb{H}}$, and so,
\begin{align*}
(-\sub)^{\alpha/2}(f_s)(x) & = \lim_{\epsilon \to 0}\frac{1}{\Gamma(k-\frac{\alpha}{2})} \int_{\epsilon}^{\infty} t^{k-\alpha/2-1}(-\triangle_{\mathbb{H}})^k H_t f_s(x) \;dt\\
&=  \lim_{\epsilon \to 0}\frac{1}{\Gamma(k-\frac{\alpha}{2})} \int_{\epsilon}^{\infty} t^{k-\alpha/2-1}(-\triangle_{\mathbb{H}})^k H_{s^2t} f(\delta_s(x)) s^{2k}\;dt\\
&= s^{\alpha} \lim_{\epsilon \to 0}\frac{1}{\Gamma(k-\frac{\alpha}{2})} \int_{\epsilon}^{\infty} (s^2t)^{k-\alpha/2-1}(-\triangle_{\mathbb{H}})^k H_{s^2t} f(\delta_s(x)) s^{2}\;dt.
\end{align*}
Then we introduce a new variable $u= s^2 t$, which yields
\begin{align}\label{eq:frac_lap}
\notag (-\sub)^{\alpha/2}(f_s)(x) &=s^{\alpha} \lim_{\epsilon \to 0}\frac{1}{\Gamma(k-\frac{\alpha}{2})} \int_{s^2 \epsilon}^{\infty} u^{-1/2}(-\triangle_{\mathbb{H}})^k H_{u} f(\delta_s(x)) du\\
&= s^{\alpha} (-\sub)^{\alpha/2}(f)(\delta_s(x))
\end{align}
Then, the second part of \eqref{eq:goal 2} follows again from \eqref{eq:goal}. \end{proof}

\section{The case $0 < \alpha < 1$}\label{alphaLessThanOne} We start by defining an auxiliary square function. Fix $0 < \alpha < 1$ and $1 < p < \infty$, and let $f \in \calD$. Write
\begin{displaymath} \calS_{\alpha}f(x) := \left(\int_{0}^{\infty} \left[\frac{1}{r^{\alpha}} \fint_{B(r)} |f(x \cdot y) - f(x)| \, dy \right]^{2} \, \frac{dr}{r} \right)^{1/2}, \end{displaymath}
where we have abbreviated $B(r) := B(0,r)$. Then, a special case of \cite[Theorem 5]{MR1828225} states that
\begin{equation}\label{form1} \|\calS_{\alpha}f\|_{L^{p}(\He^{n})} \lesssim \|(-\sub)^{\alpha/2}\|_{L^{p}(\He^{n})}, \qquad f \in \calD. \end{equation}
With \eqref{form1} in hand, the case $0 < \alpha < 1$ of Theorem \ref{mainGeneral} will (essentially) follow once we manage to control $G_{\alpha}f$ by $\calS_{\alpha}f$. To see this, first note that
\begin{displaymath} \fint_{B(x,r)} |f(y) - \langle f \rangle_{B(x,r)}| \, dy \leq \fint_{B(r)} \fint_{B(r)} |f(x \cdot y) - f(x \cdot y')| \, dy \, dy' \leq 2 \fint_{B(r)} |f(x \cdot y) - f(x)| \, dy \end{displaymath}
for all $x \in \He^{n}$ and $r > 0$. Consequently,
\begin{displaymath} G_{\alpha}f(x) = \left( \int_{0}^{\infty} \left[\tfrac{1}{r^{\alpha}} \fint_{B(x,r)} |f(y) - \langle f \rangle_{B(x,r)}| \, dy \right]^{2} \, \frac{dr}{r} \right)^{1/2} \leq 2\calS_{\alpha}f(x), \quad x \in \He^{n}, \end{displaymath}
cf.\ \cite[p. 291]{MR1828225}.
Combining this estimate with \eqref{form1} gives $\|G_{\alpha}f\|_{L^{p}(\He^{n})} \lesssim \|(-\sub)^{\alpha/2}f\|_{L^{p}(\He^{n})}$ for all $f \in \calD$, and the case $0 < \alpha < 1$ of Theorem \ref{mainGeneral} follows form Remark \ref{smoothApproximation}.

\section{The case $1 < \alpha < 2$}\label{alphaMoreThanOne}

This case will be reduced to the case $0 < \alpha < 1$. We start by recording the following result, which is a special case of \cite[Theorem 4.4.16(2)]{MR3469687}:

\begin{proposition} Let $1 < p < \infty$ and $\alpha > 0$. Then,
\begin{displaymath} \sum_{j = 1}^{2n} \|(-\sub)^{\alpha/2}X_{j}f\|_{L^{p}(\He^{n})} \lesssim \|(-\sub)^{(\alpha+1)/2}f\|_{L^{p}(\He^{n})} \qquad f \in \calD. \end{displaymath}
\end{proposition}

%\begin{remark} In \cite[Proposition 21]{MR1828225}, only the inequality "$\lesssim$" is stated, but the same argument gives the converse as well. Namely, for $1 \leq j \leq 2n$, one simply writes
%\begin{displaymath} \|X_{j}f\|_{p,\alpha} = \|(-\sub)^{(\alpha + 1)/2}(-\sub)^{-1/2}X_{j}f\|_{p,\alpha} = \|(-\sub)^{-1/2}X_{j}f\|_{p,\alpha + 1}, \end{displaymath}
%where the first equality follows again from \cite[Theorem 3.15 (iii)]{MR0494315}. Finally, one uses the fact \textcolor{cyan}{Check!} that the operator $f \mapsto (-\sub)^{-1/2}X_{j}f$ is bounded on $S^{p}_{\beta}$ for all $\beta \geq 0$. This is stated on \cite[p. 328]{MR1828225}.  \end{remark}

In addition, we will need the following $\He^{n}$-analogue of \cite[Theorem 5]{MR796440}:
\begin{proposition}\label{dorrProp} Let $1 < p < \infty$ and $0 < \alpha < 1$. Then,
\begin{displaymath} \|G_{\alpha + 1}f\|_{L^{p}(\He^{n})} \lesssim \sum_{j = 1}^{2n} \|G_{\alpha}(X_{j}f)\|_{L^{p}(\He^{n})}, \qquad f \in \calD. \end{displaymath}
\end{proposition}

We can now complete the proof of Theorem \ref{mainGeneral} in the case $1 < \alpha < 2$ as follows. Fix $1 < \alpha < 2$ and $f \in \calD$. Combining the two propositions above with the case $0 < \alpha < 1$ of Theorem \ref{mainGeneral}, we infer that
\begin{displaymath} \|G_{\alpha}f\|_{L^{p}(\He^{n})} \lesssim \sum_{j = 1}^{2n} \|(-\sub)^{(\alpha-1)/2}X_{j}f\|_{L^{p}(\He^{n})} \lesssim \|(-\sub)^{\alpha/2}f\|_{L^{p}(\He^{n})}, \end{displaymath}
as desired. So, it remains to establish Proposition \ref{dorrProp}.

\begin{proof}[Proof of Proposition \ref{dorrProp}] Let $0 < \alpha < 1$, $f \in \calD$, and recall that
\begin{displaymath} G_{\alpha + 1}f(x) = \left(\int_{0}^{\infty} \left[ \tfrac{1}{r^{\alpha + 1}} \beta_{f,1}(B(x,r)) \right]^{2} \, \frac{dr}{r} \right)^{1/2}, \end{displaymath}
where
\begin{displaymath} \beta_{f,1}(B(x,r)) = \fint_{B(x,r)} |f(y) - A^{1}_{x,r}(y)| \, dy. \end{displaymath}
Now, we define a function $\tilde{A}_{x,r} \in \calA_{1}$ so that the average of $\tilde{A}_{x,r}$ equals the average of $f$ on $B(x,r)$, and the average of $\nabla_{\He} \tilde{A}_{x,r}$ equals the average of $\nabla_{\He} f$ in a larger ball $B(x,Cr)$, where $C \geq 1$ will be chosen momentarily. Formally,
\begin{displaymath} \tilde{A}_{x,r}(z,t) := \langle f \rangle_{B(x,r)} + \langle \nabla_{\He} f \rangle_{B(x,Cr)} \cdot (z - z_{0}), \end{displaymath}
where $x = (z_{0},t_{0})$, and
\begin{displaymath} \langle \nabla_{\He} f \rangle_{B(x,Cr)} = \left( \langle X_{1} f \rangle_{B(x,Cr)},\ldots,\langle X_{2n} f \rangle_{B(x,Cr)} \right). \end{displaymath}
Then, noting that
\begin{displaymath} \langle f - \tilde{A}_{x,r} \rangle_{B(x,r)} = \langle \nabla_{\He} f \rangle_{B(x,Cr)} \cdot \fint_{B(x,r)} (z - z_{0}) \, d(z,t) = 0 , \end{displaymath}
and using Lemma \ref{lemma1} and the weak $1$-Poincar\'e inequality, see \cite{MR850547}, we obtain
\begin{align*} \beta_{f,1}(B(x,r)) & \lesssim \fint_{B(x,r)} |f(y) - \tilde{A}_{x,r}(y)| \, dy\\
& \lesssim r \fint_{B(x,Cr)} |\nabla_{\He}f(y) - \nabla_{\He} \tilde{A}_{x,r}(y)| \, dy\\
& \lesssim r \sum_{j = 1}^{2n} \fint_{B(x,Cr)} |X_{j}f(y) - \langle X_{j}f \rangle_{B(x,Cr)} | \, dy\\
& = r \sum_{j = 1}^{2n} \beta_{X_{j}f,0}(B(x,Cr)). \end{align*}
The preceding holds if $C \geq 1$ was chosen large enough. This implies that
\begin{displaymath} G_{\alpha + 1}f(x) \lesssim \sum_{j = 1}^{2n} \left( \int_{0}^{\infty} \left[\tfrac{1}{r^{\alpha}} \beta_{X_{j}f,0}(B(x,r)) \right]^{2} \, \frac{dr}{r} \right)^{1/2} = \sum_{j = 1}^{2n} G_{\alpha}(X_{j}f)(x), \quad x \in \He, \end{displaymath}
as claimed. \end{proof}

\section{Interpolation and the case $\alpha = 1$}\label{interpolation}

To handle the case $\alpha = 1$, we use \emph{complex interpolation}, see for instance \cite{MR0482275}.
%essentially use complex interpolation of the spaces $S^{p}_{\alpha}$.
In order to get the machinery started, we first observe that  $(S_{\alpha_{0}}^{p},S^{p}_{\alpha_{1}})$, $0<\alpha_0<\alpha_1<\infty$ is a \emph{compatible couple} (or \emph{interpolation pair} in the sense of Calder\'{o}n \cite{MR0167830}). That is,  $S_{\alpha_{0}}^{p}$ and $S^{p}_{\alpha_{1}}$ are both Banach spaces and they embed continuously in the space $\mathcal{S}'(\He^n)$ of tempered distributions on $\He^n$, see \cite[Theorem 4.4.3(4)]{MR3469687}.
%cf e.g. Theorem 3.5.2 in "Mathematical Methods in Physics: Distributions, Hilbert Space Operators, and .", Philippe Blanchard, Erwin Bruening
%The latter claim is immediate for $S^p_{\alpha}$ defined as in \cite[Definition 4.4.2]{MR3469687}, and by a special case of \cite[Theorem 4.4.3 (2)]{MR3469687}, this definition of $S^p_{\alpha}$ is equivalent to the one we stated in Definition \ref{d:Sobolev}.
Thus we can define the complex interpolation space $[S^{p}_{\alpha_{0}},S^{p}_{\alpha_{1}}]_{\theta}$ for $\theta \in (0,1)$.

% more precisely, we need the continuous embedding
%\begin{displaymath} S_{(1 - \theta)\alpha_{0} + \theta \alpha_{1}}^{p} \hookrightarrow [S_{\alpha_{0}}^{p},S^{p}_{\alpha_{1}}]_{\theta}, \qquad 0 < \alpha_{0} < \alpha_{1} < \infty, \: \theta \in (0,1). \end{displaymath}

\begin{lemma}\label{interpolationLemma} Let $1<p<\infty$, $\alpha_{0},\alpha_{1} \in (0,\infty)$ with $\alpha_{0} < \alpha_{1}$, fix $\theta \in (0,1)$, and let $\alpha_{\theta} := (1 - \theta)\alpha_{0} + \theta \alpha_{1}$. Then, every $f \in \calD$ satisfies
\begin{displaymath} \|f\|_{[S^{p}_{\alpha_{0}},S^{p}_{\alpha_{1}}]_{\theta}} \lesssim e^{\frac{\pi (\alpha_{1} - \alpha_{0})}{4}\sqrt{\theta(1 - \theta)}} \|f\|_{p,\alpha_{\theta}}. \end{displaymath} \end{lemma}

The proof is otherwise the same as in \cite[Lemma 34]{MR3533305}, except that the domain is $\He^{n}$ in place of $\R^{n}$, so we need to use a few results from \cite{MR0494315}, and we work with non-homogeneous Sobolev spaces. It is convenient to use a norm on $S^p_{\alpha}$ that is different from, but equivalent to, the norm in Definition \ref{d:Sobolev}. According to  \cite[Proposition 4.1]{MR0494315}, we have for $1<p<\infty$ and $\alpha \geq 0$ that
\begin{displaymath}
\|f\|_{p,\alpha} \sim \|(I- \sub)^{\alpha/2}f\|_{L^p(\He^n)},\quad f \in S^p_{\alpha},
\end{displaymath}
where $I$ denotes the identity operator on $L^{p}(\He^{n})$.

\begin{proof}[Proof of Lemma \ref{interpolationLemma}] Let $U := \{z \in \C : \Rea z \in (0,1)\}$, and fix a parameter $M > 0$ to be specified later. Fix $f \in \calD$, and define the following map $\Phi_{M} \colon \overline{U} \to S^{p}_{\alpha_{0}} + S^{p}_{\alpha_{1}}$:
\begin{displaymath} \Phi_{M}(z) := e^{M(z(z - 1) - \theta(\theta - 1))}(I-\sub)^{\frac{\alpha_{\theta} - \alpha_{0} - z(\alpha_{1} - \alpha_{0})}{2}} f, \qquad z \in \overline{U}. \end{displaymath}
To justify that $\Phi_M$ maps into  $S^{p}_{\alpha_{0}} + S^{p}_{\alpha_{1}}$, we observe that $\Phi_M(z)\in \calS$, the Schwartz class in $\He^{n}$, since $f\in \calD \subset \calS$ (see for instance the last paragraph of the proof of \cite[Corollary 4.3.16]{MR3469687}). Then it suffices to note that $\calS \subset S^p_{\alpha_0}\cap S^p_{\alpha_1} \subset S^{p}_{\alpha_{0}} + S^{p}_{\alpha_{1}}$; this is a special case of \cite[Lemma 4.4.1]{MR3469687}.

By \cite[Theorem (3.15)(iv)]{MR0494315}, $\Phi_{M}$ is an analytic $L^{p}$-valued function on $U$, satisfying $\Phi_{M}(\theta) = f$, and hence
\begin{equation}\label{form2} \|f\|_{[S^{p}_{\alpha_{0}},S^{p}_{\alpha_{1}}]_{\theta}} \leq \inf_{M > 0} \max \left\{\sup_{t \in \R} \|\Phi_{M}(it)\|_{p,\alpha_{0}},\sup_{t \in \R} \|\Phi_{M}(1 + it)\|_{p,\alpha_{1}} \right\} \end{equation}
by the definition of the norm in the complex interpolation space $[S^{p}_{\alpha_{0}},S^{p}_{\alpha_{1}}]_{\theta}$. It remains to estimate the norms on the right hand side of \eqref{form2}. Using the equation $(I-\sub)^{\alpha + \beta}f = (I-\sub)^{\alpha}(I-\sub)^{\beta} f$, see \cite[Theorem (3.15)(iii)]{MR0494315}, we find that
\begin{displaymath} (I-\sub)^{\frac{\alpha_{0} + s(\alpha_{1} - \alpha_{0})}{2}}\Phi_{M}(s - it) = e^{Ms(s - 1) + M\theta(1 - \theta) - Mt^{2} - iM(2s - 1)t}(I-\sub)^{\frac{it(\alpha_{1} - \alpha_{0})}{2}} (I-\sub)^{\frac{\alpha_{\theta}}{2}} f. \end{displaymath}
It follows that
\begin{displaymath} \| \Phi_{M}(s - it)\|_{p,\alpha_{0} + s(\alpha_{1} - \alpha_{0})} \lesssim e^{Ms(s - 1) + M\theta(1 - \theta) - Mt^{2}}\left\| (I-\sub)^{\frac{it(\alpha_{1} - \alpha_{0})}{2}} \right\|_{L^{p}(\He^{n}) \to L^{p}(\He^{n})} \|f\|_{p,\alpha_{\theta}}. \end{displaymath}
Next, we recall from \cite[Proposition (3.14)]{MR0494315} the bound
\begin{displaymath} \left\| (I-\sub)^{\frac{it(\alpha_{1} - \alpha_{0})}{2}} \right\|_{L^{p}(\He^{n}) \to L^{p}(\He^{n})} \lesssim_{p} \left|\Gamma\left(1 - \frac{it(\alpha_{1} - \alpha_{0})}{2} \right) \right|^{-1} \lesssim e^{\frac{\pi |t|(\alpha_{1} - \alpha_{0})}{4}}, \end{displaymath}
where the latter estimate follows from Stirling's formula (see also \cite[(79)]{MR3533305}). We have now reached a point corresponding to \cite[(107)]{MR3533305}; the remainder of the proof no longer uses (Heisenberg specific) results from \cite{MR0494315} and can be completed as in \cite{MR3533305}. \end{proof}

The second piece of information we need is a standard result from complex interpolation of Banach-space valued $L^{p}$ functions. Here we follow \cite{MR796440} almost verbatim.
%, using some ideas also from \cite[Proof of Theorem 40]{MR3533305}.
Let $B(1)$ denote the Kor\'{a}nyi unit ball centered at $0\in \He^n$. For $\alpha \in (0,\infty)$, we first define the Banach space $H_{\alpha}$ of functions $F \colon (0,\infty) \times B(1) \to \R$ with
\begin{displaymath} \|F\|_{H_{\alpha}} := \left(\int_{0}^{\infty} \left[ \tfrac{1}{r^{\alpha}} \fint_{B(1)} |F(y,r)| \, dy \right]^{2} \, \frac{dr}{r} \right)^{1/2} < \infty. \end{displaymath}
Then, for $1 < p < \infty$, we denote by $L^{p}(\He^{n},H_{\alpha})$ the space of functions $\Psi \colon \He^{n} \to H_{\alpha}$ with
\begin{displaymath} \|\Psi\|_{L^{p}(\He^{n},H_{\alpha})} := \left(\int_{\He^{n}} \|\Psi(x)\|^{p}_{H_{\alpha}} \, dx \right)^{1/p} < \infty. \end{displaymath}
%Topology of L^1,loc: \int_K |f-f_n| converges for every K
To apply complex interpolation, we have to verify that if $0 < \alpha_{0} < \alpha_{1} < \infty$, then $(L^{p}(\He^{n},H_{\alpha_0}),L^{p}(\He^{n},H_{\alpha_1}))$ is a compatible couple. Indeed, it follows from H\"older's inequality that
\begin{displaymath}
\int_{K_1} \int_{K_2} \int_{B(1)} |\Psi(x;y,r)| \, dy \, dr \, dx \lesssim_{\alpha,K_1,K_2,p} \|\Psi\|_{L^{p}(\He^{n},H_{\alpha})}
\end{displaymath}
for every compact set $K_1\times K_2 \subset \He^n \times (0,+\infty)$, and for all $\alpha \geq 0$ and $\Psi \in L^{p}(\He^{n},H_{\alpha})$. This shows that the Banach spaces $L^{p}(\He^{n},H_{\alpha_0})$ and $L^{p}(\He^{n},H_{\alpha_1})$ both embed continuously into $L^1_{\mathrm{loc}}(\He^n \times (0,+\infty) \times B(1))$.

As in \cite{MR796440}, we infer from the \cite[p. 107, 121]{MR0482275} that if $0 < \alpha_{0} < \alpha_{1} < \infty$ and $\theta \in (0,1)$, then
\begin{equation}\label{form3} [L^{p}(\He^{n},H_{\alpha_{0}}),L^{p}(\He^{n},H_{\alpha_{1}})]_{\theta} = L^{p}(\He^{n},[H_{\alpha_{0}},H_{\alpha_{1}}]_{\theta}) = L^{p}(\He^{n},H_{(1 - \theta)\alpha_{0} + \theta \alpha_{1}}). \end{equation}
In the proof of \eqref{form3}, there is no difference between $\R^{n}$ and $\He^{n}$. We will use \eqref{form3} for any parameters $0 < \alpha_{1} < 1 < \alpha_{2} < \infty$ and $\theta \in (0,1)$ such that $1 = (1 - \theta)\alpha_{0} + \theta \alpha_{1}$. We fix such parameters for the rest of the argument. Then, we consider the linear map
\begin{displaymath} T \colon f \mapsto Tf(x; y,r) = f(x \cdot \delta_{r}(y)) - A^{1}_{x,r}(x \cdot \delta_{r}(y)), \quad f \in S^p_{\alpha_0}+S^p_{\alpha_1}, \end{displaymath}
where $x,y \in \He^{n}$, $r > 0$, and $A^{1}_{x,r} = A^{1}_{x,r}(f)$ as in \eqref{affApprox}.
%\textcolor{cyan}{What shall be the domain of $T$? $\calS'(\He^n)$ because $S^p_{\alpha_i}$ embed continuously in $\calS'(\He^n)$}
We already know that $T$ is a bounded operator $S^{p}_{\alpha_{1}} \to L^{p}(\He^{n},H_{\alpha_{1}})$, since
\begin{align} \int_{\He^{n}} \|Tf(w)\|_{H_{\alpha_{1}}}^{p} \, dw & = \int_{\He^{n}} \left( \int_{0}^{\infty} \left[\tfrac{1}{r^{\alpha_{1}}} \fint_{B(1)} |f(x \cdot \delta_{r}(y)) - A^{1}_{x,r}(x \cdot \delta_{r}(y))| \, dy \right]^2 \, \frac{dr}{r} \right)^{p/2} \, dx \notag\\
&\label{form5} = \int_{\He^{n}} \left( \int_{0}^{\infty} \left[\tfrac{1}{r^{\alpha_{1}}} \fint_{B(x,r)} |f(y) - A^{1}_{x,r}(y)| \, dy \right]^2 \, \frac{dr}{r} \right)^{p/2} \, dx\\
& = \int_{\He^{n}} \left[ G_{\alpha_{1}}f(x) \right]^{p} \, dx \lesssim \|f\|_{p,\alpha_{1}}^{p} \notag \end{align}
by the case $1 < \alpha < 2$ of Theorem \ref{mainGeneral}. In fact, we also know that $T$ is a bounded operator $S^{p}_{\alpha_{0}} \to L^{p}(\He^{n},H_{\alpha_{0}})$. This follows from the calculation above with "$\alpha_{0}$" in place of "$\alpha_{1}$", and also plugging in the estimate
\begin{displaymath} \fint_{B(x,r)} |f(y) - A^{1}_{x,r}(y)| \, dy \lesssim \fint_{B(x,r)} |f(y) - A^{0}_{x,r}(y)| \, dy, \end{displaymath}
after line \eqref{form5} (this is immediate from Lemma \ref{lemma1}).

Now, it follows by complex interpolation that $T$ is a bounded operator
\begin{displaymath} [S^{p}_{\alpha_{0}},S^{p}_{\alpha_{1}}]_{\theta} \to [L^{p}(\He^{n},H_{\alpha_{0}}),L^{p}(\He^{n},H_{\alpha_{1}})]_{\theta} = L^{p}(\He^{n},H_{1}), \end{displaymath}
recalling \eqref{form3}. Repeating once more the calculation around \eqref{form5}, and finally using Lemma \ref{interpolationLemma}, we obtain
\begin{displaymath} \|G_{1}f\|_{L^{p}(\He^{n})} = \|Tf\|_{L^{p}(\He^{n},H_{1})} \lesssim \|f\|_{[S^{p}_{\alpha_{0}},S^{p}_{\alpha_{1}}]_{\theta}} \lesssim \|f\|_{p,1},\quad\text{for }f\in \calD. \end{displaymath}
This finishes the case $\alpha = 1$  of Theorem \ref{mainGeneral}, recalling Remark \ref{smoothApproximation} and Lemma \ref{l:inhomog_homog}.
The proof of Theorem \ref{mainGeneral} is complete.

\section{Extension to $L^{p}$-mean $\beta$-numbers}\label{extensions}

In this section, we consider an extension of Theorem \ref{main}, briefly mentioned in Section \ref{exAndAppl}, which is analogous to the one discussed at the end of Dorronsoro's paper, \cite[Section 5]{MR796440}. To avoid over-indexing, we slightly re-define our notation for this last section. For $1 \leq q < \infty$, we write
\begin{displaymath} \beta_{q}(B(x,r)) := \beta_{f,1,q}(B(x,r)) := \left( \fint_{B(x,r)} |f(y) - A_{x,r}(y)|^{q} \, dy \right)^{1/q}, \end{displaymath}
where
\begin{displaymath} A_{x,r} = A_{x,r}(f) \in \calA_{1}. \end{displaymath}
So, $\beta_{1}(B(x,r))$ corresponds to $\beta_{f,1}(B(x,r))$ in the previous notation.

\begin{thm}\label{mainGeneralised} Write $Q := 2n + 2$. Let $1 < p < \infty$ and $q \geq 1$ with
\begin{displaymath} \begin{cases} q < \frac{pQ}{Q - p}, & \text{if } 1 < p \leq 2,\\ q < \frac{2Q}{Q - 2}, & \text{if } 2 \leq p < \infty. \end{cases} \end{displaymath}
Let $f \in L^{p}(\He^{n})$ with $\nabla_{\He} f \in L^{p}(\He^{n})$. Then,
\begin{equation}\label{form13} \left( \int_{\He^{n}} \left( \int_{0}^{\infty} \left[\tfrac{1}{r}\beta_{q}(B(x,r)) \right]^{2} \, \frac{dr}{r} \right)^{p/2} \, dx \right)^{1/p} \lesssim \|\nabla_{\He} f\|_{L^{p}(\He^{n})}. \end{equation}
\end{thm}

\begin{remark} The case $q = 1$ is just Theorem \ref{main}. The general case follows from the argument given in \cite[Section 5]{MR796440}, and there is virtually no difference between $\R^{n}$ and $\He^{n}$ here: the idea is to demonstrate that the left hand of \eqref{form13} is bounded by $\|G_{1}f\|_{L^{p}(\He^{n})}$, which is then further bounded by the $L^{p}$-norm of $\nabla_{\He} f$ by Theorem \ref{main}. We will give the details for the reader's convenience. \end{remark}

We begin by claiming that if $x \in \He^{n}$, $r > 0$, then
\begin{equation}\label{form6} |f(y) - A_{x,r}(y)| \lesssim \int_{0}^{4r} \beta_{1}(B(y,s)) \, \frac{ds}{s}, \qquad \text{for a.e. } y \in B(x,r). \end{equation}
To see this, it suffices to establish that
\begin{equation}\label{form7} \fint_{B(y,2^{-n}r)} |f(z) - A_{x,r}(z)| \, dz \lesssim \int_{2^{-n}r}^{4r} \beta_{1}(B(y,s)) \, \frac{ds}{s} \end{equation}
for all $y$ in the open ball $B(x,r)$, and for all $n \in \N$ sufficiently large (depending on $y$). Then \eqref{form6} will follow by Lebesgue's differentiation theorem. To derive \eqref{form7}, pick $y \in B(x,r)$ and $n \in \N$ so large that $B(y,2^{-n}r) \subset B(x,r)$. Then, start with the following estimate:
\begin{align*} \fint_{B(y,2^{-n}r)} |f(z) - A_{x,r}(z)| \, dz & \leq \fint_{B(y,2^{-n}r)} |f(z) - A_{y,2^{-n}r}(z)| \, dz\\
& \quad + \sum_{k = 1}^{n + 1} \fint_{B(y,2^{-n}r)} |A_{y,2^{k - 1 - n}r}(z) - A_{y,2^{k - n}r}(z)| \, dz\\
& \quad + \fint_{B(y,2^{-n}r)} |A_{y,2r}(z) - A_{x,r}(z)| \, dz =: I_{1} + I_{2} + I_{3}. \end{align*}
Here
\begin{equation}\label{form8} I_{1} = \beta_{1}(B(y,2^{-n}r)) \lesssim \int_{2^{-n}r}^{2^{-n + 1}r} \beta_{1}(B(y,s)) \, \frac{ds}{s}, \end{equation}
applying Corollary \ref{betaMonotonicity}. To treat $I_{2}$, recall that $A_{x,s} := A_{x,s}(f) \in \calA_{1}$. In general, we will write $A_{x,s}(g)$ for the element of $\calA_{1}$ corresponding to $g \in L^{1}_{loc}(\He^{n})$. In particular, since $A_{y,2^{k - n}r} \in \calA_{1}$, we have
\begin{displaymath} A_{y,2^{k - 1 - n}r}(A_{y,2^{k - n}r}) = A_{y,2^{k - n}r}. \end{displaymath}
Hence, for $1 \leq k \leq n + 1$,
\begin{align*} \fint_{B(y,2^{-n}r)} |A_{y,2^{k - 1 - n}r}(z) - A_{y,2^{k - n}r}(z)| \, dz & = \fint_{B(y,2^{-n}r)} |A_{y,2^{k - 1 - n}r}[f - A_{y,2^{k - n}r}](z)| \, dz\\
& \leq \|A_{y,2^{k - 1 - n}r}[f - A_{y,2^{k - n}r}]\|_{L^{\infty}(B(y,2^{k - 1 - n}r))}\\
& \lesssim \fint_{B(y,2^{k - 1 - n}r)} |f - A_{y,2^{k - n}r}(z)| \, dz,   \end{align*}
using \eqref{claim} in the final inequality. Then, to complete the treatment of $I_{2}$, it remains to note that
\begin{displaymath} \fint_{B(y,2^{k - 1 - n}r)} |f - A_{y,2^{k - n}r}(z)| \, dz \lesssim \beta_{1}(B(y,2^{k - n}r)) \lesssim \int_{2^{k - n}r}^{2^{k + 1 - n}r} \beta_{1}(B(y,s)) \, \frac{ds}{s}, \end{displaymath}
using Corollary \ref{betaMonotonicity} again. Finally, to estimate $I_{3}$, note that
\begin{displaymath} A_{y,2r} = A_{x,r}(A_{y,2r}), \end{displaymath}
and hence
\begin{align*} \fint_{B(y,2^{-n}r)} |A_{y,2r}(z) - A_{x,r}(z)| \, dz & = \fint_{B(y,2^{-n}r)} |A_{x,r}[A_{y,2r} - f](z)| \, dz\\
& \stackrel{\eqref{claim}}{\lesssim} \fint_{B(x,r)} |A_{y,2r}(z) - f(z)| \, dz\\
& \lesssim \fint_{B(y,2r)} |A_{y,2r}(z) - f(z)| \, dz = \beta_{1}(B(y,2r)).  \end{align*}
In the application of \eqref{claim}, we used the assumption that $B(y,2^{-n}r) \subset B(x,r)$. Finally,
\begin{displaymath} \beta_{1}(B(y,2r)) \lesssim \int_{2r}^{4r} \beta_{1}(B(y,s)) \, \frac{ds}{s} \end{displaymath}
by Corollary \ref{betaMonotonicity}. Summing the estimates above for $I_{1},I_{2},I_{3}$ completes the proof of \eqref{form7}.

As a corollary of \eqref{form6}, we infer the following inequality, which is an analogue of \cite[(11)]{MR796440}: for $x \in \He^{n}$ and $r > 0$,
\begin{equation}\label{form9} |f(y) - A_{x,r}(y)| \lesssim \int_{0}^{4r} \beta_{1}(B(y,s)) \, \frac{ds}{s} \lesssim \int_{0}^{4r} \fint_{B(y,s)} \beta_{1}(B(z,2s)) \, dz \, \frac{ds}{s}  \end{equation}
for Lebesgue almost all $y \in B(x,r)$. To obtain the second inequality, use Corollary \ref{betaMonotonicity} once more.

From this point on, one can follow the proof presented after \cite[(11)]{MR796440} quite literally. Fix, first,
\begin{displaymath} 1 < p \leq 2 \quad \text{and} \quad 1 \leq q < \frac{pQ}{Q - p}. \end{displaymath}
Then, choose some $1 < w < p$ and $0 < \beta < 1$ such that
\begin{equation}\label{form10} q = \frac{wQ}{Q - \beta w}. \end{equation}
We will apply the fact that the fractional Hardy-Littlewood maximal function
\begin{displaymath} M_{\beta}g(y) := \sup_{s > 0} \left\{ s^{\beta} \fint_{B(y,s)} |g(z)| \, dz \right\} \end{displaymath}
maps $L^{w}(\He^{n}) \to L^{q}(\He^{n})$ boundedly, when $q,w$ are related as in \eqref{form10}. This fact holds generally in $Q$-regular metric measure spaces, see for example \cite[Theorem 4.1]{MR3108871}. It now follows from \eqref{form9}, Minkowski's integral inequality, and the boundedness of $M_{\beta}$ that
\begin{align*} r^{Q/q}\beta_{q}(B(x,r)) & \lesssim \int_{0}^{4r} s^{-\beta} \left( \int_{B(x,r)} M_{\beta}(\beta_{1}(B(\cdot, 2s))\chi_{B(x,5r)})(y)^{q} \, dy \right)^{1/q} \, \frac{ds}{s}\\
& \lesssim \int_{0}^{4r} s^{-\beta} \left( \int_{B(x,5r)} \beta_{1}(B(z,2s))^{w} \, dz \right)^{1/w} \, \frac{ds}{s}\\
& \lesssim \int_{0}^{4r} s^{-\beta} r^{Q/w} \mathcal{M}_{w}(\beta_{1}(B(\cdot,2s))) \, \frac{ds}{s}, \end{align*}
where $\mathcal{M}_{w}$ stands for the maximal function $\mathcal{M}_{w}(g) = (M(|g|^{w}))^{1/w}$. Consequently, recalling from \eqref{form10} that $Q/w - Q/q = \beta$, we arrive at
\begin{displaymath} \beta_{q}(B(x,r)) \lesssim r^{\beta} \int_{0}^{4r} s^{-\beta - 1} \mathcal{M}_{w}(\beta_{1}(B(\cdot,2s))(x) \, ds. \end{displaymath}
Next, noting that $3 - 2\beta > 1$ and using Hardy's inequality in the form
\begin{displaymath} \int_{0}^{\infty} r^{3 - 2\beta} \left(\int_{0}^{r} g(s) \, ds \right)^{2} \, dr \lesssim \int_{0}^{\infty} r^{2\beta - 1} g(r)^{2} \, dr, \end{displaymath}
see \cite[Theorem 330]{MR0046395}, we obtain
\begin{align} \left( \int_{0}^{\infty} \left[\tfrac{1}{r}\beta_{q}(B(x,r))\right]^{2} \, \frac{dr}{r} \right)^{1/2} & \lesssim \left(\int_{0}^{\infty} r^{2\beta - 3} \left[ \int_{0}^{4r} s^{-\beta - 1} \mathcal{M}_{w}(\beta_{1}(B(\cdot,2s)))(x) \, ds \right]^{2} \, dr \right)^{1/2}\notag \\
&\label{form12} \lesssim \left(\int_{0}^{\infty} \mathcal{M}_{w}(\tfrac{1}{r}\beta_{1}(B(\cdot,2r)))(x)^{2} \, \frac{dr}{r} \right)^{1/2}. \end{align}
Finally, following \cite{MR796440} verbatim, the estimate
\begin{equation}\label{form11} \left( \int_{\He^{n}} \left( \int_{0}^{\infty} \left[\tfrac{1}{r}\beta_{q}(B(x,r))\right]^{2} \, \frac{dr}{r} \right)^{p/2} \, dx \right)^{1/p} \lesssim \left(\int_{\He^{n}} \left( \int_{0}^{\infty} \left[\tfrac{1}{r} \beta_{1}(B(x,r)) \right]^{2} \, \frac{dr}{r} \right)^{p/2} \, dx \right)^{1/p} \end{equation}
can be inferred from the Fefferman-Stein vector-valued maximal function inequality, namely
\begin{equation}\label{FS} \bigg( \int_{\He^{n}} \bigg[\sum_{j \in \Z} M(g_{j})^{u}(x) \bigg]^{v/u} \, dx \bigg)^{1/v} \lesssim \bigg( \int_{\He^{n}} \bigg[\sum_{j \in \Z} g_{j}(x)^{u} \bigg]^{v/u} \, dx \bigg)^{1/v}. \end{equation}
Here $(g_{j})_{j \in \Z}$ is a family of functions $\He^{n} \to \R$, and $u,v \in (1,\infty)$. The inequality \eqref{FS} is valid generally in $Q$-regular metric measure spaces, see \cite[Theorem 1.2]{MR2542655}. To infer \eqref{form11} from \eqref{form12}, one needs to apply \eqref{FS} to functions of the form
\begin{displaymath} g_{j}(x) = 2^{j}\beta_{1}(B(x,2^{-j + 2})) \end{displaymath}
and exponents $u = 2/w > 1$ and $v = p/w > 1$, noting (by Corollary \ref{betaMonotonicity}) that
\begin{displaymath} M_{w}(\tfrac{1}{r}\beta_{1}(B(\cdot,2r))) \lesssim M_{w}(g_{j}), \qquad r \in [2^{-j - 1},2^{-j}]. \end{displaymath}
Once \eqref{form11} has been established, the case $1 \leq p \leq 2$ of Theorem \ref{mainGeneralised} is a corollary of Theorem \ref{main} (since the right hand side of \eqref{form11} is precisely $\|G_{1}f\|_{L^{p}(\He^{n})}$).

The case $p \geq 2$ is similar. Indeed, if $p \geq 2$ and $1 \leq q < 2Q/(Q - 2)$, we may choose $1 < w < 2$ and $0 < \beta < 1$ such that \eqref{form10} holds. Then, the arguments after \eqref{form10} can be repeated. Finally, the choices of exponents $u := 2/w > 1$ and $v := p/w > 1$ also remain valid in the application of the Fefferman-Stein inequality. The proof of Theorem \ref{mainGeneralised} is complete.

\section{Application: Vertical vs. horizontal Poincar\'e inequalities}\label{HvsVSection}

As a corollary of Theorem \ref{mainGeneralised}, we derive the following \emph{vertical vs. horizontal Poincar\'e inequality} originally due to Lafforgue and Naor \cite[Theorem 2.1]{MR3273443} (the case $p = q = 2$ was earlier obtained by Austin, Naor, and Tessera \cite{MR3095705}; see also \cite[Remark 43]{NY}):

\begin{thm}\label{HvsV} Let $1 < p \leq 2$, and let $f \in L^{p}(\He^{n})$ with $\nabla_{\He} f \in L^{p}(\He^{n})$. Then,
\begin{equation}\label{HvsVIneq} \left(\int_{0}^{\infty} \left[ \int_{\He^{n}} \left(\frac{|f(x) - f(x \cdot (\mathbf{0},t))|}{\sqrt{t}} \right)^{p} \, dx \right]^{2/p} \, \frac{dt}{t} \right)^{1/2} \lesssim \|\nabla_{\He} f\|_{L^{p}(\He^{n})}. \end{equation}
\end{thm}

Here we have denoted by $(\mathbf{0},t)$ the point $(0,\ldots,0,t) \in \R^{2n} \times \R$ with $t \in \R$. In \cite{MR3273443}, the target of $f$ is allowed to be a much more general Banach space than $\R$, and the "$2$" in \eqref{HvsVIneq} can also be a more general exponent $q \geq 2$.

\begin{proof}[Proof of Theorem \ref{HvsV}] Fix $1 < p \leq 2$ and $t > 0$. We first claim that
\begin{equation}\label{form14} \int_{\He^{n}} \left(\frac{|f(x) - f(x \cdot (\mathbf{0},t))|}{\sqrt{t}} \right)^{p} \, dx \lesssim\int_{\He^{n}}  \left[\tfrac{1}{\sqrt{t}} \cdot \beta_{p}(B(x,C_{2}\sqrt{t}))\right]^{p} \, dx, \end{equation}
if $C_{2} \geq 1$ is a sufficiently large constant. Here $\beta_{p}$ is the $L^{p}$-based $\beta$-number, as defined at the head of the previous section. Let $\calB_{\sqrt{t}}$ be a collection of balls "$B$" of radius $\sqrt{t}$ whose union covers $\He^{n}$, and such that the concentric balls "$\hat{B}$" of radius $C_{1}\sqrt{t}$ have bounded overlap for some constant $1 \leq C_{1} < C_{2}$ to be determined shortly. For $B \in \calB_{\sqrt{t}}$, let $A_{\hat{B}} \in \calA_{1}$ be the affine function from the definition of $\beta_{p}(\hat{B})$. Since
\begin{displaymath} A_{\hat{B}}(x) = A_{\hat{B}}(x \cdot (\mathbf{0},t)), \qquad x \in \He^{n}, \: t > 0, \end{displaymath}
we may deduce \eqref{form14} as follows, assuming that $1 \leq C_{1} < C_{2}$ are appropriately chosen:
\begin{align*} \int_{\He^{n}} \left(\frac{|f(x) - f(x \cdot (\mathbf{0},t))|}{\sqrt{t}} \right)^{p} \, dx & \leq \sum_{B \in \calB_{\sqrt{t}}}  \int_{B} \left(\frac{|f(x) - f(x \cdot (\mathbf{0},t))|}{\sqrt{t}} \right)^{p} \, dx\\
& \lesssim \sum_{B \in \calB_{\sqrt{t}}} \bigg[ \int_{B} \left(\frac{|f(x) - A_{\hat{B}}(x)|}{\sqrt{t}} \right)^{p}\\
&\qquad \qquad \quad + \left(\frac{|f(x \cdot (\mathbf{0},t)) - A_{\hat{B}}(x \cdot (\mathbf{0},t))}{\sqrt{t}} \right)^{p} \, dx \bigg]\\
&\lesssim \sum_{B \in \calB_{\sqrt{t}}} \int_{\hat{B}} \left(\frac{|f(x) - A_{\hat{B}}(x)|}{\sqrt{t}} \right)^{p} \, dx\\
& = \sum_{B \in \calB_{\sqrt{t}}} \left[ \tfrac{1}{\sqrt{t}} \cdot \beta_{p}(\hat{B})\right]^{p}|\hat{B}| \lesssim \int_{\He^{n}}  \left[\tfrac{1}{\sqrt{t}} \cdot \beta_{p}(B(x,C_{2}\sqrt{t}))\right]^{p} \, dx. \end{align*}
In the last inequality we implicitly used Corollary \ref{betaMonotonicity}. Now, the inequality \eqref{HvsVIneq} follows from Minkowski's integral inequality and Theorem \ref{mainGeneralised} with $q = p$:
\begin{align*} & \left(\int_{0}^{\infty} \left[ \int_{\He^{n}} \left(\frac{|f(x) - f(x \cdot (\mathbf{0},t))|}{\sqrt{t}} \right)^{p} \, dx \right]^{2/p} \, \frac{dt}{t} \right)^{1/2}\\
& \qquad \stackrel{\eqref{form14}}{\lesssim} \left(\int_{0}^{\infty} \left[ \int_{\He^{n}} \left[\tfrac{1}{\sqrt{t}} \cdot \beta_{p}(B(x,C_{2}\sqrt{t}))\right]^{p} \, dx \right]^{2/p} \, \frac{dt}{t} \right)^{1/2}\\
& \qquad \: \leq \left( \int_{\He^{n}} \left[ \int_{0}^{\infty} \left[\tfrac{1}{\sqrt{t}} \cdot \beta_{p}(B(x,C_{2}\sqrt{t}))\right]^{2} \, \frac{dt}{t} \right]^{p/2} \, dx \right)^{1/p} \lesssim \|\nabla_{\He}f\|_{L^{p}(\He^{n})}, \end{align*}
making the change of variables $\sqrt{t} \mapsto r$ before applying Theorem \ref{mainGeneralised}. The proof of Theorem \ref{HvsV} is complete. \end{proof}

\bibliographystyle{plain}
\bibliography{references}

\end{document}